\documentclass[12pt]{article}
 \usepackage{amsmath, latexsym, amsfonts, amssymb, amsthm, amscd, stmaryrd,setspace}

 \theoremstyle{plain}
 \newtheorem{theorem}{Theorem}
 \newtheorem*{theorem*}{Theorem}
 \newtheorem{lemma}{Lemma}
 \newtheorem{corollary}{Corollary}

\newtheorem*{fact*}{Fact}
 \newtheorem*{assumption*}{Assumption}

 \renewcommand{\tilde}{\widetilde}          
 \DeclareMathSymbol{\leqslant}{\mathalpha}{AMSa}{"36} 
 \DeclareMathSymbol{\geqslant}{\mathalpha}{AMSa}{"3E} 
 \DeclareMathSymbol{\eset}{\mathalpha}{AMSb}{"3F}     
 \renewcommand{\leq}{\;\leqslant\;}                   
 \renewcommand{\geq}{\;\geqslant\;}                   

 \newcommand{\R}{\mathbb{R}}
 \newcommand{\Z}{\mathbb{Z}}

\newcommand{\RWRE}{random walks in random environment}
\newcommand{\IP}{\mathbb{P}}
\newcommand{\E}{\mathbb{E}}

\newcommand{\8}{\infty}

 \title {Random walk delayed on percolation clusters}
\author{Francis~Comets$^1$\thanks{Partially supported
by CNRS (UMR 7599 ``Probabilit{\'e}s et Mod{\`e}les Al{\'e}atoires'')} \and
Fran\c cois Simenhaus$^{2*}$}

\begin{document}

\maketitle

{\footnotesize
\noindent $^{~1}$Universit{\'e} 
Paris Diderot -- Paris 7, UFR de Math{\'e}matiques,
Case 7012,
75205 Paris Cedex 13
\\
\noindent e-mail: \texttt{comets@math.jussieu.fr},
\noindent url: \texttt{http://www.proba.jussieu.fr/$\sim$comets}

\smallskip
\noindent $^{~2}$Universit{\'e} 
Paris Diderot -- Paris 7, UFR de Math{\'e}matiques,
Case 7012,
75205 Paris Cedex 13
\\
\noindent e-mail: \texttt{simenhaus@math.jussieu.fr}

}

 \begin{abstract} We study a continuous time random walk on the $d$-dimensional
lattice, subject to 
a drift and an attraction to large clusters of a subcritical 
Bernoulli site percolation. We find two distinct regimes:
a ballistic one, and a subballistic one
taking place when the 
   attraction is strong enough. We  identify the speed in the former case, 
and the algebraic rate of escape in the latter case. Finally, we 
discuss the  diffusive behavior in the case of
zero drift  and weak attraction.
 \end{abstract}

\noindent
{\bf Short Title:} Random walk on percolation clusters
\\[.3cm]{\bf Keywords:} Random walk in random environment, subcritical percolation, 
anomalous transport, anomalous diffusion, environment seen from the particle, coupling
\\[.3cm]{\bf AMS 2000 subject classifications:} 60K37; 60F15, 60K35, 82D30

 \section{Model and results} \label{sec:model}

Consider the graph of nearest neighbors on $\mathbb{Z}^d$,  $d \geq 1$, 
and write $x \sim y$ when $\|x-y\|_1=1$. Here, $\|\cdot\|_1$ is the
$\ell_1$-norm, though  $|\cdot |$ denotes the Euclidean norm.

An environment is an element $\omega$ of $\Omega=\{0,1\}^{\mathbb{Z}^d}$. 
Environments are used to construct the independent identically distributed 
(i.i.d.) Bernoulli site percolation on the lattice.
We consider the product $\sigma$-field on $\Omega$ and for $p\in(0,1)$, the probability $\mathbb{P}=\mathcal{B}(p)^{\otimes\mathbb{Z}^d}$, where $\mathcal{B}(p)$ denotes the Bernoulli law with parameter $p$. A site $x$ in $\mathbb{Z}^d$ is said open if $\omega_x=1$, and closed otherwise. 
Consider the open connected  components (so-called clusters) in the percolation
graph. The cluster of an open site
$x \in \mathbb{Z}^d$ is the union of $\{x\}$ with the 
set of all $y \in \mathbb{Z}^d$ which are connected to $x$ by a path
with all vertices open.  The cluster of a closed site is empty.
We denote by $C_x$  the cardinality of the cluster  of $x$.

It is well known that there exists a critical $p_c=p_c(d)$ such that for $p<p_c$, $\mathbb{P}$-almost surely, all connected open components (clusters) of $\omega$ are finite, though for $p > p_c$, there a.s. exists an infinite cluster. 
Moreover, it follows from \cite{Aizenman-Barsky87}, \cite{Menshikov86}
that, in the first case, the  clusters size has an exponential tail:
 For any $p<p_c$, there exists $\xi=\xi(p)>0$ such that for all $x$,
\begin{equation*}
 \lim_{n\to \infty}\frac{1}{n}\ln\mathbb{P}(C_x\geq n)=-\xi\;.
\end{equation*}
In this paper, we fix $p < p_c$.
Let $\ell=(\ell_k; 1 \leq k \leq d)$ be a unit vector, $\lambda$ and $\beta$
two non-negative number.
For every environment $\omega$, let  $P_{\omega}$ be the law of the continuous
time Markov chain $Y=(Y_t)_{t\geq 0}$ on ${\mathbb{Z}^d}$ starting at $0$ with generator $L$ given for continuous bounded functions $f$ by
\begin{equation*}
 Lf(x)= K
\sum_{e \sim 0}
e^{\lambda \ell \cdot e -\beta C_x}
\Big[ f(x+e) -f(x) \Big]\;,
\end{equation*}
where we chose the normalizing constant $K$  as 
$K=\big(\sum_{e \sim 0}e^{\lambda \ell \cdot e} \big)^{-1}$ for simplicity.
Given $\omega$, define the  measure  $\mu$ on $\mathbb{Z}^d$ by
\begin{equation} \label{eq:mu}
\mu(x)=e^{2\lambda \ell \cdot x +\beta C_x}\;.
\end{equation}
The random measure $\mu$ combines a shift in 
the direction $\ell$ together
with an attraction to large clusters.
Observe that the process $Y$ admits $\mu$ as  invariant, reversible measure.
The markovian time evolutions of $\mu$ are of natural 
interest in the context of \RWRE. They describe random walks which 
have a tendency to live on 
large clusters, the attraction becoming stronger as $\beta$ is 
increased. The  isotropic case, $\lambda=0$, has been
considered in \cite{Popov-Vachkovskaia05} with a different, discrete-time 
dynamics. There, the authors proved that the walk is diffusive for small 
$\beta$, and 
subdiffusive for large $\beta$. The investigation of slowdowns in the 
anisotropic case is then natural. 
In \cite{Shen02}, a random resistor network is considered with a
invariant reversible  measure of the form 
$C(x,\omega) e^{2 \lambda \ell \cdot x}$
where the random field $(C(x,\omega);x \in \Z^d)$ is stationary  ergodic and 
bounded away from 0 and $+ \8$: in this case, the \RWRE \ is ballistic
for all positive $\lambda$.

The study of a general dynamics
in the presence of a drift contains many difficult questions, and the 
advantage of the  particular
process $Y$ considered here is that we can push the analysis  farther.
We could as well handle the discrete time analogous of $Y$, i.e. the \RWRE\
with geometric holding times instead of exponential ones, which falls in the 
class of marginally nestling walks in the standard classification 
(e.g., \cite{ZeitouniSF}).
The Markov process
$Y$ can also be described with it skeleton and its jump rates.
The  skeleton $X=(X_n)_{n\in \mathbb{N}}$ is defined as 
the sequence of distinct consecutive locations
visited by $Y$. Then, $X$ is a discrete time Markov chain with transition 
probabilities $\tilde{P}$, given for  $x\in\mathbb{Z}^d$ and $e \sim 0$ by
\begin{equation*}
 \forall x\in\mathbb{Z}^d,\ \forall e \sim 0,\quad 
\tilde{P}(X_{n+1}=x+e|X_n=x)=
\frac{e^{\lambda \ell\cdot e}}{\sum_{e' \sim 0}e^{\lambda \ell \cdot e'}}=: 
{\tilde p}_e\;,
\end{equation*}
and $\tilde{P}(X_{n+1}=y|X_n=x)=0$ if $y$ is not a nearest neighbor of $x$.
This Markov chain is simple, since $X$ is the random walk on $\mathbb{Z}^d$ with drift 
\begin{equation} \label{eq:drift}
d(\lambda)
= \frac{1}{ \sum_{k=1}^d \cosh (\lambda 
\ell_k)}\; \Big( \sinh(\lambda \ell_k)\Big)_{1\leq k \leq d}.
\end{equation}
It is plain that for the random walk, 
\begin{equation} \label{eq:LLNX}
\frac{X_n}{n} \longrightarrow d(\lambda)\qquad \tilde{P}-a.s., 
\end{equation}
so directional transience is clear and
the law of large number for $Y$ boils down to
the study of the clock process which takes care of the real time for jumps.
As can be seen from formula (\ref{eq:S_n}),
the process considered here is a generalization of the so-called random
walk in a random scenery, or the random walk subordinated to a renewal process,
which are used as effective models for anomalous diffusions. The difference is
essentially that the
environment field (i.e., the mean holding times)
has here some short-range correlations due to the percolation. It is
also related to the trap model considered in
the analysis of the aging phenomenon introduced in \cite{Bouchaud92}:
the aging of this model has been studied in details, see
\cite{Benarous-Cerny05} for a recent review.

For a fixed $\omega$, $P_{\omega}$ is called the quenched law and we define the annealed law $P$ by
\begin{equation*}
 P=\mathbb{P} \times P_{\omega}.
\end{equation*}
Of course, statements which hold $P$-a.s., equivalently hold
$ P_{\omega}$-a.s. for $\IP$-a.e. environment.

Finally, we stress that we assume $d \geq 1$ in this paper.
The case $d=1$ is special since the critical  threshold  $p_c(1)=1$.
Moreover, specific techniques are available in one dimension, 
e.g. \cite{ZeitouniSF} for a survey, however we will stick as much as possible 
to techniques applying for all $d$.

Our first result is the law of large numbers.
\begin{theorem}
\label{theorem:lln} {\bf (Law of large numbers)}
 For any $\lambda \geq 0$ and any $\beta \geq 0$,
\begin{equation*}
 \frac{Y_t}{t}\xrightarrow[t\to +\infty]{} v(\lambda,\beta),\quad P-a.s.,
\end{equation*}
where 
\begin{equation} \label{eq:vitesse}
v(\lambda,\beta)=\Big(\mathbb{E} e^{\beta C_0}\Big)^{-1} d(\lambda)\;.
\end{equation}
In particular, $v(\lambda,\beta)=0$ if $\beta>\xi$ or $\lambda=0$ 
though $v(\lambda,\beta)\cdot \ell >0$ if $\beta<\xi$ and $\lambda\neq 0$.
\end{theorem}
As in the case $\lambda=0$ considered in \cite{Popov-Vachkovskaia05},
slowdowns occur for large disorder intensity $\beta$, when the walk gets
trapped on large percolation clusters. This behavior is
reminiscent of the biased random walk on the supercritical percolation
infinite cluster  \cite{Sznitman03},  \cite{Berger-Gantert-Peres03} where ballistic or subballistic regimes take place according to the parameters values. 
The slowdowns in our paper have a similar nature to those in some 
one dimensional
\RWRE, see \cite{Solomon75}, \cite{Kesten-Kozlov-Spitzer75} and \cite{Sinai82}.
Moreover, as in the one dimensional case, 
we obtain here explicit values for the rate of escape, a rather unusual fact
in larger dimension.
More drastic (logarithmic) slowdowns were also found for an unbiased walker 
in a moon craters
landscape in \cite{Durrett86}, \cite{Bramson-Durrett88}, or diffusions in 
random potentials \cite{Mathieu94}, but in these models the behavior at 
small disorder  is qualitatively different from the behavior without disorder.

The next result contains extra information on the subballistic behavior. 

\begin{theorem}
\label{theorem:vitesse}  {\bf (Subballistic regime)}
Let $\beta \geq \xi$.
\begin{enumerate}
	 \item\label{item:partie1} For any $d\geq 1$ and $\lambda>0$,
\begin{equation*}
\frac{\ln|Y_t|}{\ln t}\xrightarrow[t\to +\infty]{}\frac{\xi}{\beta} \qquad 
P-a.s.
	\end{equation*}
	\item\label{item:partie2} If $\lambda=0$, for any $d\geq 2$ we have
\begin{equation*}
 \limsup_{t\to +\infty}\frac{\ln|Y_t|}{\ln t} = \frac{\xi}{2\beta} \qquad 
P-a.s.
\end{equation*}
	\item\label{item:partie3} If $d=1$ and $\lambda=0$ we have
\begin{equation*}
 \limsup_{t\to +\infty}\frac{\ln|Y_t|}{\ln t} = \frac{1}{2}(\frac{\beta}{2\xi}+\frac{1}{2})^{-1} \qquad 
P-a.s.
\end{equation*}

\end{enumerate}
\end{theorem}

Hence, the spread of the \RWRE \ scales algebraically with time in all cases.
Note that in the isotropic case $\lambda=0$, the slowdown is larger for
$d=1$ than for $d \geq 2$. This will appear in the proof as a consequence of 
the strong recurrence of the simple random walk $X$ in the
one-dimensional case.
Note that our results are only in the logarithmic scale, though
the scaling limit has been  obtained for the isotropic trap model,
in dimension $d=1$ (e.g.,  \cite{Benarous-Cerny05}), and $d\geq 1$
 \cite{Benarous-Cerny06} with limit given, if the disorder is strong,
by the time change of a Brownian motion by the inverse of a stable subordinator
(fractional kinetics). Though we believe that the  scaling limit of
our model without drift ($\lambda=0$) is the same, we could not get finer
results because of the presence of correlations
in the medium. Moreover, the case of a drift $\lambda \neq 0$ has not
been considered in the literature, except for $d=1$ with renormalization group
arguments \cite{Monthus04}.

To complete the picture, we end by the diffusive case.
(Recall that $\beta < \xi$ is sufficient for $\E(e^{\beta C_0})<\8$.)

\begin{theorem}  
\label{theorem:diff} 
{\bf (Diffusive case regime)}
Assume $\lambda=0$, and $\E(e^{\beta C_0})<\8$. 
Then, we have a quenched invariance principle for the rescaled process
 $Z^{\epsilon}=(Z^{\epsilon}_t)_{t \geq 0}$, 
$Z^{\epsilon}_t={\epsilon}^{1/2} Y_{{\epsilon}^{-1}t}$:
For almost every $\omega$, as $\epsilon \searrow 0$,
the family of processes $Z^{\epsilon}$
converges in law under $P_\omega$
to the $d$-dimensional 
Brownian motion with diffusion matrix $\Sigma=\big(d \times 
\E(e^{\beta C_0})\big)^{-1}
I_d$. Moreover, 
\begin{equation} \label{eq:derniercas}
 \limsup_{t\to +\infty}\frac{\ln|Y_t|}{\ln t} = \frac{1}{2} \qquad 
a.s.
\end{equation}
\end{theorem}

For the proof of our results 
we will take the point of view of the environment seen from the
walker. It turns out that the ``static'' environmental distribution is 
invariant for the dynamics. Hence the environment is always at equilibrium.

The paper is organized as follows.
In the next section, we introduce the basic ingredients for our analysis
and we prove the law of large numbers of Theorem \ref{theorem:lln}. 
The last section is devoted to the subballistic regime and 
contains the proofs
 of Theorem \ref{theorem:vitesse}
and \ref{theorem:diff}.


 \section{Preliminaries and the proof of Theorem \ref{theorem:lln}} 
\label{sec:ensuite}

For $x \in \mathbb{Z}^d$, $T^x$ will denote the space shift with vector $x$.
We will consider also the  time shift $\theta$.
\medskip

\noindent
{\bf Skeleton and clock process of $Y$.} The sequence $(S_n; n \geq 0)$
of jump times of the Markov process $Y$ with right-continuous paths
is defined by $S_0=0<S_1<S_2<\ldots$, 
$Y_t=Y_{S_n}$ for $t \in [S_n,S_{n+1})$, $Y_{S_{n+1}}\neq Y_{S_n}$. 
The skeleton of $Y$ is
the sequence $X$ given by $X_n=Y_{S_n}, n\geq 0$.
As mentioned above, the skeleton $X$
of $Y$ is the simple random walk with drift.
For any $x$ in $\mathbb{Z}^d$, the jump rate of $(Y_t)_{t\geq 0}$ at $x$ 
is $e^{-\beta C_x}$. Hence the time $S_n$ of the $n$-th jump is
the sum of $n$ independent random variables with exponential distribution with
mean $e^{\beta C_{X_i}}, i=1,\ldots n$. This means that the sequence 
 $\mathcal{E}=
(\mathcal{E}_i)_{i \in \mathbb{N}}$, with $ \mathcal{E}_i=e^{-\beta C_{X_i}}
(S_{i+1}-S_{i})$,
is  a sequence of i.i.d. exponential variables with mean $1$, with 
$\mathcal{E}$ and $X$ independent.
 The law of this sequence will be denoted by $Q$ ($Q=\mathcal{E}xp(1)^{\otimes \mathbb{N}}$, with $\mathcal{E}xp(1)$ the mean $1$, exponential law).
For any $n$ in $\mathbb{N}$, the time $S_n$ of the $n$-th jump is given by
\begin{equation} \label{eq:S_n}
 S_n=\sum_{i=0}^{n-1} \mathcal{E}_i e^{\beta C_{X_i}}.
\end{equation}
This sequence can be view as a step function $S_t:=S_{[t]}$,
 where $[\cdot]$ is the integer part, and we also define
its generalized inverse $S^{-1}$:
for any $t\geq 0$, 
\begin{equation*}
S^{-1}(t)=n \iff  S_n\leq t <S_{n+1}\;.
\end{equation*}
We observe that $S_n \to \8$ as $n \to \8$ $P_{\omega}$-a.s. for all $\omega$, 
making the function $S^{-1}$ defined on the whole of $\R_+$. 
Then, $P_{\omega}$-a.s.,
\begin{equation} \label{eq:Y=}
X_{S^{-1}(t)} = Y(t)\;,\quad \forall t\geq 0\;.
\end{equation}
and therefore, the process $S^{-1}$ is called the clock process.

Conversely, let $\mathcal{E}, X$ and $\omega$ be independent, with 
distribution $Q, \tilde P$ and $\IP$ respectively, defined on some new 
probability space. Then, fixing $\lambda$
and viewing $\beta$ as a parameter, by (\ref{eq:S_n}) and (\ref{eq:Y=}) 
we construct, on this  new probability space, a coupling of the processes 
$Y=Y^{(\beta)}$ for all $\beta \in \R$. The coupling has the properties that 
the skeleton is the same for all $\beta$, and that the clock processes
are such that for $\beta \geq \beta'$ and $t \geq 0$,
\begin{equation} \label{eq:coupling}
S^{-1}(\beta; t) \leq S^{-1}(\beta'; t).
\end{equation}

\medskip

\noindent
{\bf The environment seen from the walker.}
Depending on the time being discrete or continuous, we consider
the processes $(\tilde{\omega}_n)_{n \in \mathbb{N}}$ and 
$(\hat{\omega}_t)_{t \geq 0}$ defined by
\begin{equation*}
 \tilde{\omega}_n=T^{X_n}\omega\;,\quad 
\hat{\omega}_t = T^{Y_t}\omega = \tilde{\omega}_{S^{-1}(t)}
\end{equation*}
for $n \geq 0, t \geq 0$. We start with the case of discrete time.
\begin{lemma}
\label{lemma:ergodique}
 Under $P$, $(\tilde{\omega}_i)_{i\in \mathbb{N}}$ 
is a stationary ergodic Markov chain.
The same holds for 
$(\tilde{\omega}_i, \mathcal{E}_i)_{i\in \mathbb{N}}$.
\end{lemma}

\begin{proof} [Proof of Lemma \ref{lemma:ergodique}]
As $(\mathcal{E}_i)_{i\in\mathbb{N}}$ is an i.i.d. sequence of variables independent of  $\tilde{\omega}$, it is enough to prove Lemma \ref{lemma:ergodique} for the process $(\tilde{\omega}_i)_{i\in \mathbb{N}}$.
Under $P$ (resp $P_{\omega}$)  $(\tilde{\omega}_i)_{i\in \mathbb{N}}$ is markovian  with 
transition kernel $R$ defined for any bounded function $f$ by
\begin{equation*}
 Rf(\omega)=\sum_{e\sim 0}{\tilde p}_ef(T^e \omega)
\quad \forall \omega \in \Omega,
\end{equation*}
and initial distribution $\mathbb{P}$ (resp $\delta_{\omega}$).
The transitions of $(\tilde{\omega}_i)_{i\in \mathbb{N}}$ does not depend on $\omega$ like those of $X$ and, in this sense, the sequence 
is itself a random walk.
Since $\mathbb{P}$ is invariant by translation,
\begin{equation*}
 E[f(\tilde{\omega}_1)]=\int\sum_{e \sim 0}{\tilde p}_ef(T^e\omega)d\mathbb{P}=\sum_{e \sim 0}{\tilde p}_e\int f(T^e\omega)d\mathbb{P}=\mathbb{E}[f(\omega)],
\end{equation*}
showing that  $\mathbb{P}$ is an invariant measure for  
$(\tilde{\omega}_i)_{i\in \mathbb{N}}$. 

We will use $\mathcal{F}$ to denote the product $\sigma$-field on $\Omega^{\mathbb{N}}$, and for any $k\geq 0$, $\mathcal{F}_k$ will denote the $\sigma$-field generated by the $k$ first coordinates.  Note that $\theta$ is measurable and preserves the law of $\tilde{\omega}$ under $P$.
We have to prove that the invariant $\sigma$-field $\Sigma:=\{A\in \mathcal{F}, 1_A(\tilde{\omega})=1_A(\theta\tilde{\omega}),\ P\textrm{-a.s.}\}$ is trivial. Let $Y$ be a $\Sigma$-mesurable bounded random variable on $\Omega^{\mathbb{N}}$, we have to show that it is $P$-a.s. constant.

Define for all $\omega$ in $\Omega$, $h_Y(\omega):=E_{\omega}[Y]$. We will study this function with standard arguments e.g. chapter 17.1.1 of \cite{Meyn_Tweedie}. Using Markov property and the $\theta$-invariance of $Y$, we can show that,
\begin{equation}
 \label{eq:martingale}
h_Y(\tilde{\omega}_k)=E[Y|\mathcal{F}_k]\qquad\forall k\in\mathbb{N},\ P\textrm{-a.s.}
\end{equation}
As a consequence, under $P$, $(h_Y(\tilde{\omega}_k))_{k\geq 0}$ is both a stationary process and an a.s. convergent martingale, and hence it is a.s. constant. In particular,
\begin{equation*}
 Y=h_Y(\tilde{\omega}_0)\qquad P\textrm{-a.s.,}
\end{equation*}
what means that $Y$ can be consider as a function of the first coordinate alone.
The next step is to show that $h_Y$ is $\mathbb{P}$-a.s. harmonic, that is
\begin{equation*}
 Rh_Y(\tilde{\omega}_0)=h_Y(\tilde{\omega}_0),\quad P\textrm{-a.s.}
\end{equation*}
It is a consequence of the following computation,
\begin{align*}
 Rh_Y(\tilde{\omega}_0)&=E[h_Y(\tilde{\omega}_1)|\mathcal{F}_0]\qquad P\textrm{-a.s.}\\
			&=E[E[Y|\mathcal{F}_1]|\mathcal{F}_0]\qquad P \textrm{-a.s.}\\
			&=h_Y(\tilde{\omega}_0)\qquad P\textrm{-a.s.},
\end{align*}
where the second equality is true because of (\ref{eq:martingale}).
We will now show that $Y$ is invariant by translation in space. By invariance 
of $\mathbb{P}$ and harmonicity of $h_Y$, it is true that 
\begin{equation*}
 \sum_{e \sim 0}\int {\tilde p}_e(Y-Y\circ T^e)^2 d\mathbb{P}=0.
\end{equation*}
For every $e$ neighbour of $0$, ${\tilde p}_e>0$, and the previous equation implies that, $\mathbb{P}$ almost surely $Y=Y\circ T^e$ for any $e\sim 0$.
Together with the ergodicity of $\mathbb{P}$, this shows that $Y$ is $P$-a.s.
constant, and completes the proof.
\end{proof}

As a consequence of Lemma \ref{lemma:ergodique} and Birkhoff's ergodic theorem, for any fucntion 
$f$ in $L_1(\Omega^{\mathbb{N}})$ (or $f$ non negative),
\begin{equation*}
 \frac{1}{n}\sum_{k=0}^{n-1} f(\theta^k\tilde{\omega})\xrightarrow{n\to+\infty}E[f]\qquad P\textrm{- a.s.}
\end{equation*}

Now, we turn to the time continuous case, and we consider the empirical 
distribution $\frac{1}{t} \int_0^t \delta_{\hat \omega_s} ds$
of the environment seen from the walker up to time $t$. 
Our next result is
a law of large numbers for this random probability measure.
For small $\beta$, the empirical distribution converges to some limit 
$\mathbb{P}^0$, which is then an invariant measure for 
$(\hat \omega_t)_{t \geq 0}$.

\begin{corollary}
\label{corollary:empirique}
 If $\beta<\xi$ then $P$-almost surely, the empirical distribution
of the environment seen from the walker,
$\frac{1}{t} \int_0^t \delta_{\hat \omega_s} ds$, 
converges weakly to $\mathbb{P}^0$ defined by $d\mathbb{P}^0=\frac{e^{\beta C}}{\mathbb{E}[e^{\beta C}]}d\mathbb{P}$.
\end{corollary}
\begin{proof}[Proof of Corollary \ref{corollary:empirique}]
We need to show that ${t}^{-1} \int_0^t f({\hat \omega_s}) ds
\to \int f d\mathbb{P}^0$ as $t \to \8$, for all 
real bounded continuous function $f$ on $\Omega$. Since $e^{\beta C_{0}}$
is integrable when $\beta < \xi$, this follows from
the convergence along the sequence $t=S_n$, $n \to \8$. 
By (\ref{eq:S_n}), this is equivalent to
$$
\frac{n^{-1}\sum_{i=0}^{n-1}  \mathcal{E}_i e^{\beta C_{X_i}}
f({\tilde{\omega}_i})}{n^{-1}\sum_{i=0}^{n-1}  \mathcal{E}_i e^{\beta C_{X_i}}}
 \longrightarrow  \int_\Omega f d\mathbb{P}^0\;, \quad
n \to \8.
$$
We first study the $P$-almost sure convergence of the denominator, i.e.
of $n^{-1}{S_n}$. Define the real function $g$ on 
$(\mathbb{R}^{\mathbb{N}},\Omega^{\mathbb{N}})$
\begin{equation*}
g: ((\mathcal{E}_i)_{i\in \mathbb{N}},(\tilde{\omega}_{i})_{i\in\mathbb{N}})
\mapsto \mathcal{E}_0e^{\beta C_0(\tilde{\omega}_0)}
\end{equation*}
and note that $C_{X_n}=C_0(\tilde{\omega}_n)$.
Applying Lemma \ref{lemma:ergodique} and the ergodic theorem to $(\tilde{\omega},\mathcal{E})$ and to the non negative function $g$, we obtain that $n^{-1}{S_n}$ 
converges $P$-almost surely to $\mathbb{E}[e^{\beta C_0}]$.
The numerator can be studied with the same arguments, and we obtain the claim
since for $\beta < \xi$ both limits are finite.
\end{proof}

With this in hand, we can easily complete the 

\begin{proof}[Proof of Theorem \ref{theorem:lln}]
Write 
\begin{equation*}
 \frac{Y_t}{t}=\frac{X_{S^{-1}(t)}}{S^{-1}(t))}\;\frac{S^{-1}(t)}{S(S^{-1}(t))}
\;\frac{S(S^{-1}(t))}{t}\;.
\end{equation*}
Recall from (\ref{eq:LLNX}) that the first factor in the right-hand side
converge almost surely to $d(\lambda)$ as $t \to \8$. In the proof of
Corollary \ref{corollary:empirique} we have shown that 
${S(S^{-1}(t))}/{S^{-1}(t)} \to \mathbb{E}[e^{\beta C_0}]$ a.s. for 
 $\beta < \xi$, but clearly the result remains true for all $\beta$ (the limit 
is infinite for  $\beta > \xi$). For the last  factor in the right-hand side
we simply observe that 
\begin{equation}
\label{equation:encadrement}
 \frac{S(S^{-1}(t))}{S(S^{-1}(t)+1)} \leq \frac{S(S^{-1}(t))}{t} \leq 1\;,
\end{equation}
yielding that  ${S(S^{-1}(t))}/{t}$ converges $P$-almost surely to $1$
if $\mathbb{E}[e^{\beta C_0}]<\8$: in this case, we then conclude that 
${Y_t}/{t}$ converges $P$-almost surely to $v(\lambda, \beta)$ given by
(\ref{eq:vitesse}). 

In the case $\mathbb{E}[e^{\beta C_0}]=\8$, we just use the right inequality in (\ref{equation:encadrement}) to obtain the $P$-almost surely convergence of ${Y_t}/{t}$ to $v(\lambda, \beta)=0$.
\end{proof}

\section{Subballistic regime, and the proofs of Theorem \ref{theorem:vitesse}
and \ref{theorem:diff}}
\label{sec:Subballistic}

We start with a few auxiliary results.

\begin{lemma}
\label{lemma:range}
Assume $d \geq 2$ or $\lambda > 0$. 
Then, for any $\epsilon>0$, there exists $\alpha>0$ 
such that $P$-almost surely, we eventually have
\begin{equation*}
 \sharp\Big\{i\leq n,\ C_{X_i}>(\frac{1}{\xi}-\epsilon)\ln n 
\Big\}\geq n^{\alpha}
\end{equation*}
with the notation $ \sharp A$ for the cardinality of a set $A$.
\end{lemma}
\begin{proof}[Proof of Lemma \ref{lemma:range}]
Define the range $R_n$ as the number of points visited by $(X_i)_{i\in\mathbb{N}}$ during the first $n$ steps. For $\lambda >0$, there exists a constant $c_1>0$ such that $\tilde{P}$-almost surely eventually $R_n>c_1 n$.
For $\lambda=0$ and $d>2$, it is well known (see chapter 21 of \cite{Revesz}) that there exists a constant $c_2$ such that  $\tilde{P}$-almost surely eventually $R_n >c_2 \frac{n}{\ln n}$ (when $d\geq 3$, the walk is transient and the correct order of $R_n$ is $n$). In all cases, there exists a constant $c_3>0$ such that under the assumptions of  Lemma \ref{lemma:range}, we have $\tilde{P}$-almost surely, eventually, $R_n>c_3 \frac{n}{\ln n}$.
For a fixed $n$ in $\mathbb{N}$, we define recursively the time $T^n_i$ by
\begin{align*}
 T_0^n&=0,\\
 T_i^n&=\inf\{T_{i-1}^n <k\leq n,\ |X_k-X_{T_j^n}|>2(\frac{1}{\xi}-\epsilon)\ln n,\ \forall j<i\}\ \quad \forall i\geq 1,\\
\inf\emptyset&=+\infty.
\end{align*}
Note that the balls with center $X_{T_j^n}$ and radius 
$({\xi}^{-1}-\epsilon)\ln n$ are pairwise disjoint, and define 
$K_n$ the number of such balls, i.e.
\begin{equation*}
 K_n=\max\{i\geq 0:\ T_i^n<+\infty\}
\end{equation*}
As the cardinality of those ball is $c_4\ln^d n$ (for some $c_4>0$), it follows from the previous discussion on the range that $\tilde{P}$-almost surely, eventually, $K_n>c\frac{n}{\ln^{d+1}n}$, where $c$ denotes a positive constant. 
From now on we fix a path $(X_i)_{i \leq n}$ such that  $K_n>c\frac{n}{\ln^{d+1}n}$. Then,
\begin{align*}
& \!\!\!\!\!\!\!\!\!\!
\mathbb{P}\Big(\sharp\big\{i\leq K_n,\ C_{X_{T^n_i}}\leq(\frac{1}{\xi}-\epsilon)\ln n\big\}\geq K_n-n^{\alpha}\Big)\\
&= \mathbb{P}\Big( \exists I \subset \{1,\ldots K_n\}, \sharp I = K_n - 
[n^\alpha]: \forall i \in I,\; C_{X_{T^n_i}}\leq(\frac{1}{\xi}-\epsilon)\ln n\
\Big)\\
&\leq \sum_{I \subset \{1,\ldots K_n\}, \sharp I = K_n - 
[n^\alpha]} \mathbb{P}\Big(\forall i \in I,\; C_{X_{T^n_i}}\leq(\frac{1}{\xi}-\epsilon)\ln n\
\Big)
\end{align*}
For all $j$ such that $0\leq j\leq K_n-n^{\alpha}$, $B^n_i$ denotes the ball with center $X_{T^n_{i}}$ and  radius $(\frac{1}{\xi}-\epsilon)\ln n$. The event $\{C_{X_{T^n_{i}}}\leq (\frac{1}{\xi}-\epsilon)\ln n)\}$ is $\sigma\{\omega_x, x\in B^n_i\}$ measurable. As the balls $B_i^n$ are disjoint and the environment is i.i.d.,
\begin{align*}
&\!\!\!\!\!\!\!\!\!\!
\mathbb{P}\Big(\sharp\big\{i\leq K_n,\ C_{X_{T^n_i}}\leq(\frac{1}{\xi}-\epsilon)\ln n\big\}\geq K_n-n^{\alpha}\Big)\\
&\leq {K_n \choose n^{\alpha}} \big(1-\mathbb{P}(C_0>(\frac{1}{\xi}-\epsilon)\ln n)\big)^{K_n-n^{\alpha}}\\
&\leq c_5 n^{n^{\alpha}}\big(1-n^{-(1-\epsilon\xi)+o(1)}\big)^{c\frac{n}{\ln^{d+1}n}}, 
\end{align*}
for some suitable constant $c_5 > 0$.
We now choose $\alpha<\min(1,\epsilon\xi)$, so that 
$$ \sum_n
\mathbb{P}\Big(\sharp\big\{i\leq K_n,\ C_{X_{T_i}}\leq(\frac{1}{\xi}-\epsilon)\ln n\big\}\geq K_n-n^{\alpha}\Big)< \8
$$ 
We conclude using Borel-Cantelli's lemma.
\end{proof}

\begin{lemma}
\label{lemma:limsup}
Assume $\beta > \xi$. For $d \geq 2$ or $\lambda>0$, we have
 $\liminf_n \frac{\ln S_n}{\ln n}\geq\frac{\beta}{\xi}$, $P$-almost surely.
\end{lemma}
\begin{proof}[Proof of Lemma \ref{lemma:limsup}]
Let $\eta$ be a positive real number. 
With $\epsilon:=\eta/\beta$,
from Lemma \ref{lemma:range}, there exists $\alpha>0$ such that  $\tilde{P}\otimes\mathbb{P}$-almost surely, there exists a natural number $N=N(X,\omega)$ such that for $n>N$, the set $I=\{i\leq n,\ C_{X_i}>
(\frac{1}{\xi}-\epsilon)\ln n \}$ has cardinality  $\sharp I \geq n^{\alpha}$.
For $n>N$,
\begin{align*}
Q(S_n<n^{{\beta}/{\xi}-\eta}) &\leq 
Q(\mathcal{E}_i e^{\beta C_{X_i}} <n^{\beta /\xi-\eta},\; i \in I)\\
&<Q(\mathcal{E}_1 e^{\beta C_{X_i}} < n^{\beta/\xi-\eta})^{n^{\alpha}}\\
&<Q(\mathcal{E}_1 < n^{\beta \epsilon -\eta})^{n^{\alpha}}\\
&= (1-e^{-1})^{n^{\alpha}}.
\end{align*}
From previous inequality, we obtain that $Q(S_n<n^{{\beta}/{\xi}-\eta})$ is the general term of a convergent series and we can use Borel-Cantelli's Lemma to conclude.
\end{proof}

\begin{lemma}
\label{lemma:liminf} 
Assume $\beta > \xi$. For $d \geq 1$ and $\lambda \geq 0$, we have
$P$-almost surely, $\limsup_n \frac{\ln S_n}{\ln n}\leq\frac{\beta}{\xi}$.
\end{lemma}
\begin{proof}[Proof of Lemma \ref{lemma:liminf}]
For any $\alpha$ in $(0,1)$, by subadditivity we have $(u+v)^\alpha \leq
u^\alpha + v^\alpha$ for all positive $u,v$, and then
\begin{equation*}
 S_n^{\alpha}\leq \sum_{i=1}^{n}\mathcal{E}_i^{\alpha}e^{\alpha\beta C_{X_i}}.
\end{equation*}
Now, define the function $f_{\alpha}$
\begin{eqnarray*}
f_{\alpha}:&(\mathbb{R}^{\mathbb{N}},\Omega^{\mathbb{N}})&\rightarrow \mathbb{R}\\ 
  &((\mathcal{E}_i)_{i\in\mathbb{N}},(\tilde{\omega}_{i})_{i\in\mathbb{N}})&\rightarrow \mathcal{E}_0^{\alpha}e^{\alpha\beta C_0(\tilde{\omega}_0)}.
\end{eqnarray*}
Applying Lemma \ref{lemma:ergodique} and the ergodic theorem to $(\tilde{\omega},\mathcal{E})$ with the non negative function $f_{\alpha}$, we obtain that for any $\alpha$ such that $\alpha \beta<\xi$, 
$$
 \limsup_{n\to +\infty} \frac{ S_n^\alpha}{n}
\leq 
 \lim_{n\to +\infty} \frac{\sum_{i=1}^{n}\mathcal{E}_i^{\alpha}
e^{\alpha\beta C_{X_i}} }{n} = E_Q(\mathcal{E}_1^{\alpha}) \times
\E( e^{\alpha\beta C_0})
<\8
$$
almost surely. Therefore,
\begin{equation*}
 \limsup_{n\to +\infty} \frac{\ln S_n}{\ln n}<\frac{1}{\alpha}.
\end{equation*}
Since $\alpha$ is arbitrary in $(0,\xi/\beta)$, the proof is complete.
\end{proof}

The two following lemmas deal with the one dimensional case. Notice that when $d=1$, for all $n>0$,
\begin{equation*}
 \mathbb{P}(C\geq n)=p\sum_{k=0}^{n-1}p^kp^{n-1-k}=np^n,
\end{equation*}
and as a consequence $\xi=-\ln p$.
\begin{lemma}
\label{lemma:bornesup}
Assume $\beta>\xi$. For $d=1$ and $\lambda=0$, we have $P$-almost surely,
$\limsup_n \frac{\ln S_n}{\ln n}\leq \frac{\beta}{2\xi}+\frac{1}{2}$.
\end{lemma}

\begin{proof}[Proof of Lemma \ref{lemma:bornesup}]
Here we need to relabel our sequence of exponential variables $(\mathcal{E}_i;
i \geq 0)$. For $y\in\mathbb{Z},k\in\mathbb{N}$, define $\mathcal{E}_{y,k})$
by
$$
\mathcal{E}_{y,k} = (\mathcal{E}_n \quad {\rm with}\; i\; {\rm such \ that}
\quad X_i=y, \sharp \{j: 0 \leq j \leq i, X_j=y\}=k\;,
$$
i.e. the exponential corresponding to the $k$-th passage at $y$. These new 
variables are a.s. well defined when $d=1$ and $\lambda=0$, and 
it is not difficult to see that the sequence $(\mathcal{E}_{y,k})_{y\in\mathbb{Z},k\in\mathbb{N}}$ is i.i.d. with mean $1$ exponential distribution, 
and independent of $X$ and of $\omega$.
The number of visits of the walk to a site $y$ at time $n$ will be denoted by $\theta(n,y)$. We can rewrite $S_n$ in the following way,
\begin{equation} \label{eq:newS_n}
S_n=\sum_{i=0}^{n-1}e^{\beta C_{X_i}}\mathcal{E}_i=\sum_{y\in \mathbb{Z}}e^{\beta C_y}\left(\sum_{k=0}^{\theta(n,y)-1}\mathcal{E}_{y,k}\right).
\end{equation}
Notice that for any $\eta>0$, $\tilde{P}$-almost surely for $n$ large enough, $\theta(n,y)=0$ for $y>n^{\frac{1}{2}+\eta}$ (see for example  Theorem 5.7 p44 in \cite{Revesz}).
As a consequence, we obtain that for any positive $\alpha<1$, $\tilde{P}$-almost surely for $n$ large enough,
\begin{equation*}
S_n^{\alpha}\leq \sum_{y=-n^{-\frac{1}{2}+\eta}}^{n^{\frac{1}{2}+\eta}}e^{\alpha \beta C_y}\left(\sum_{k=0}^{\theta(n,y)-1}\mathcal{E}_{y,k}\right)^{\alpha}.
\end{equation*}
Here and below, the sum $\sum_{y=a}^b$ with  real numbers $a<b$, ranges over 
all $y \in   \mathbb{Z}$ with $a \leq y \leq b$.
Notice now that for any $\nu>0$, $\tilde{P}$-almost surely for $n$ large enough, $\sup\{\theta(n,y),y\in\mathbb{Z}\}<n^{\frac{1}{2}+\nu}$(see for example  Theorem 11.3 p118 in \cite{Revesz}) and we obtain for such $n$,
\begin{equation}
\label{equation:majoration}
\frac{1}{2n^{\frac{1}{2}+\eta}n^{(\frac{1}{2}+\nu)\alpha}}S_n^{\alpha}\leq
\frac{1}{2n^{\frac{1}{2}+\eta}}\sum_{y=-n^{-\frac{1}{2}+\eta}}^{n^{\frac{1}{2}+\eta}}e^{\alpha \beta C_y}(\frac{1}{n^{\frac{1}{2}+\nu}}\sum_{k=0}^{n^{\frac{1}{2}+\nu}}\mathcal{E}_{y,k})^{\alpha}.
\end{equation}
For any $y$ in $\mathbb{Z}$ and $n$ in $\mathbb{N}$, we define $\displaystyle u_{y,n}=\frac{1}{n^{\frac{1}{2}+\nu}}\sum_{k=0}^{n^{\frac{1}{2}+\nu}}\mathcal{E}_{y,k}$.
Fix $\mu>0$, according to the large deviation principle for i.i.d. sequences, there exists $I_{\mu}>0$ such that, for any $y$ in $\mathbb{Z}$ and any $n$ in $\mathbb{N}$,
\begin{equation*}
{Q}(|u_{y,n}-1|>\mu)\leq e^{-I_{\mu} n^{\frac{1}{2}+\nu}}.
\end{equation*}
Using the independance of the $(\mathcal{E}_{y,k})_{y\in\mathbb{Z},k\in\mathbb{N}}$, it is easy to check that ${Q}(\exists y \in [-n^{\frac{1}{2}+\eta},n^{\frac{1}{2}+\eta}], |u_{y,n}-1|>\mu)$ is the general term of a convergent series and using Borel-Cantelli's lemma we obtain that ${Q}$-almost surely, for $n$ large enough and for any $-n^{\frac{1}{2}+\eta}<y<n^{\frac{1}{2}+\eta}$,
\begin{equation}
\label{equation:uniforme}
|u_{y,n}-1|<\mu.
\end{equation}
From the ergodicity of the environment, it is true that $\mathbb{P}$-almost surely,
\begin{equation}
\label{equation:ergodicity}
\frac{1}{2n^{\frac{1}{2}+\eta}}\sum_{y=-n^{-\frac{1}{2}+\eta}}^{n^{\frac{1}{2}+\eta}}e^{\alpha \beta C_y}\xrightarrow{n\to +\infty}{}\mathbb{E}[e^{\alpha \beta C}].
\end{equation}
Using now (\ref{equation:majoration}),(\ref{equation:uniforme}) and (\ref{equation:ergodicity}), we obtain that for any $\alpha<\xi/\beta$, there exists $M<+\infty$ such that, $P$-almost surely for $n$ large enough,
\begin{equation*}
S_n<Mn^{\frac{1}{2\alpha}+\frac{\eta}{\alpha}+\frac{1}{2}+\nu}.
\end{equation*}
Since the last inequality is true for $\eta$ and $\mu$ arbitrary small and $\alpha$ arbitrary close to $\xi/\beta$, the proof is complete.
\end{proof}


\begin{lemma}
\label{lemma:borneinf}
Assume $\beta>\xi$. For $d=1$ and $\lambda=0$, we have $P$-almost surely,
$\liminf_{n \to \8} \frac{\ln S_n}{\ln n}\geq \frac{\beta}{2\xi}+\frac{1}{2}$.
\end{lemma}

\begin{proof}[Proof of Lemma \ref{lemma:borneinf}]
Let $\eta$ and $\nu$ be two positive real numbers.
We recall two facts used in Lemma \ref{lemma:bornesup}, $\tilde{P}$-almost surely and for $n$ large enough,
\begin{itemize}
 \item  $\theta(n,y)=0$ for any $y\geq n^{\frac{1}{2}+\frac{\eta}{2}}$,
 \item  $\sup\{\theta(n,y),y\in\mathbb{Z}\}<n^{\frac{1}{2}+\frac{\nu\xi}{4}}$.
\end{itemize}
As a consequence of those two facts, $\tilde{P}$-almost surely, for $n$ large enough, at least $n^{\frac{1}{2}-\frac{\nu\xi}{4}}$ sites are visited more than $n^{\frac{1}{2}-\eta}$ times, we will denote the set of those sites by $O_n$. Fix now a path $(X_i)_{i\geq 0}$ such that for all $n\geq 0$, $\sharp O_n\geq n^{\frac{1}{2}-\frac{\nu\xi}{4}}$.
As in the proof of Lemma \ref{lemma:limsup}, we can choose a familly of $\alpha_n:=\frac{n^{\frac{1}{2}-\frac{\nu\xi}{4}}}{\frac{1}{2}(\frac{1}{\xi}-\nu)\ln n}$ points $(y_i)_{i\leq \alpha_n}$ in $O_n$ such that the intervals $(I_i)_{i\leq \alpha_n}$ centered in $(y_i)_{i\leq \alpha_n}$ and of length $\frac{1}{2}(\frac{1}{\xi}-\nu)$ are disjoint. If all sites of an intervall are open, it will be said open, otherwise it will be said closed. Using the fact that the $(I_i)_{i\leq \alpha_n}$ are disjoint, we obtain that,
\begin{align*}
\mathbb{P}(I_i\ \textrm{is closed, for all}\ i\leq \alpha_n)&\leq(1-n^{-\frac{1}{2}(1-\nu\xi)+o(1)})^{\alpha_n}\\
&\leq e^{-n^{\frac{\nu\xi}{4}+o(1)}}.
\end{align*}
As a consequence of Borell-Cantelli's lemma we obtain that $P$-almost surely, for $n$ large enough, there exists at least one site visited more than $n^{\frac{1}{2}-\eta}$ times and that belongs to a cluster of size greater than $\frac{1}{2}(\frac{1}{\xi}-\nu)\ln n$, we will note this site $\tilde{y}_n$, and therefore,
\begin{equation*}
S_n\geq \sum_{i=0}^{n^{\frac{1}{2}-\eta}}n^{\frac{\beta}{2\xi}-\nu\beta}\mathcal{E}_{\tilde{y}_n,i}.
\end{equation*}
Using the large deviation upper bound similarly to the lines 
below (\ref{equation:majoration}), 
we obtain from the last inequality that $P$-almost surely, for $n$ large enough,
\begin{equation*}
S_n\geq \frac{1}{2}n^{\frac{1}{2}+\frac{\beta}{2\xi}-\nu\beta-\eta}.
\end{equation*}
Since $\nu$ and $\eta$ can be choosen arbitrary small, this last inequality 
ends the proof.
\end{proof}


\begin{proof}[Proof of Theorem \ref{theorem:vitesse}]
We first assume that $\beta > \xi$. 
From Lemma \ref{lemma:limsup} and Lemma \ref{lemma:liminf}, we know that under assumptions of parts \ref{item:partie1} or \ref{item:partie2} of Theorem \ref{theorem:vitesse}, 
\begin{equation*}
 \lim_{n\to +\infty}\frac{\ln S_n}{\ln n}=\frac{\beta}{\xi}\qquad P-{\rm a.s.}
\end{equation*}
From the inequalities
\begin{equation*}
 \frac{\ln S(S^{-1}(t))}{\ln S^{-1}(t)}\leq \frac{\ln t}{\ln S^{-1}(t)}
< \frac{\ln S(S^{-1}(t)+1)}{\ln S^{-1}(t)},
\end{equation*}
we deduced that $P$-almost surely,
\begin{equation*}
 \lim_{t \to +\infty}\frac{\ln t}{\ln S^{-1}(t)}=\frac{\beta}{\xi}.
\end{equation*}
Applying the same arguments as above, we deduce from Lemma \ref{lemma:bornesup} and Lemma \ref{lemma:borneinf} that under assumptions of part \ref{item:partie3} of Theorem \ref{theorem:vitesse},
\begin{equation*}
 \lim_{t \to +\infty}\frac{\ln t}{\ln S^{-1}(t)}=\frac{\beta}{2\xi}+\frac{1}{2},\qquad P-{\rm a.s.}
\end{equation*}
Write now,
\begin{equation*}
 \frac{\ln |Y_t|}{\ln t}=\frac{\ln |X_{S^{-1}(t)}|}{\ln S^{-1}(t)}\frac{\ln S^{-1}(t)}{\ln t}.
\end{equation*}
To conclude in the case $\beta > \xi$, note that under assumptions of part \ref{item:partie1}, $\frac{\ln |X_{n}|}{\ln n}$ converges $\tilde{P}$-almost surely to $1$ and under assumption of part \ref{item:partie2} and \ref{item:partie3}, $\tilde{P}$-almost surely, $\limsup_{n\to +\infty} \frac{\ln |X_{n}|}{\ln n}= \frac{1}{2}$ by the law of iterated logarithm.

To extend the results to the border case $\beta = \xi$, we use the property
(\ref{eq:coupling}) of the coupling, which implies that the long-time limit of
$\frac{\ln |Y_t|}{\ln t}$ is non-increasing in $\beta$. This completes the proof of part 1 with $\beta = \xi$. Now, we will prove independently  
$\limsup_{n\to +\infty}  \frac{\ln |Y_t|}{\ln t}=1/2$ (\ref{eq:derniercas})
below.  Again, by the monotonicity of the coupling, this ends the proof of parts 2 and 3 with $\beta = \xi$.
\end{proof}


\begin{proof}[Proof of Theorem \ref{theorem:diff}]
First observe that when $\lambda=0$, 
$$
f(Y_t)-\int_0^t  (2d)^{-1} e^{-\beta C_{Y_s}} \sum_{e \sim 0}
\Big[ f(Y_s+e) -f(Y_s) \Big]ds
$$
is a $P_\omega$-martingale for $f$ continuous and bounded. Then, 
for all $\omega$, the process $Y$ is a square integrable martingale
under the quenched law $P_\omega$. Its bracket is the unique 
process $\langle Y \rangle$ taking its values in the space of nonnegative 
symmetric $d \times d$ matrices such that  $Y_t Y_t^* -\langle Y \rangle_t$ 
is a martingale and $\langle Y \rangle_0=0$. We easily compute 
$$
\langle Y \rangle_t= 
\int_0^t e^{-\beta C_{Y_s}}ds  \times
d^{-1} I_d
$$
By  Corollary \ref{corollary:empirique}, we see that the bracket 
$Z^{\epsilon}$ is such that, for all $t \geq 0$,
\begin{eqnarray} \nonumber
\langle Z^{\epsilon} \rangle_t &=&
{\epsilon} \langle Y \rangle_{{\epsilon}^{-1}t} \\ \nonumber
 &=& 
\epsilon \int_0^{\epsilon^{-1} t} e^{-\beta C_{Y_s}}ds \times
d^{-1} I_d\\ \nonumber
&\longrightarrow &
t \Sigma  \qquad {\rm as\ }\epsilon \searrow 0
\end{eqnarray}
$P$-a.s., and then in $P_\omega$-probability for a.e. $\omega$. Let us fix
such an $\omega$, and use the law $P_\omega$.
Since the martingale $ Z^{\epsilon}$ has jumps of size $\epsilon^{-1/2}$
tending to 0
and since its bracket converges to a deterministic limit, it is well known 
(e.g. Theorem VIII-3.11 in \cite{Jacod-Shiryaev87}) that the sequence
  $(Z^{\epsilon}, \epsilon >0)$ converges to the centered Gaussian process
with variance $t \Sigma$, yielding the desired 
invariance principle under  $P_\omega$.

We now prove (\ref{eq:derniercas}). Since $\lambda=0$ we have $\limsup_n
\ln |X_n|/\ln n = 1/2, \tilde P$-a.s., and since $\E e^{\beta C_0}<\8$
it holds a.s. $\lim_t \ln S^{-1}(t)/\ln t = 1$. This implies the claim.
\end{proof}

{\bf Concluding remarks:} 
(i) Part 2 of  Theorem \ref{theorem:vitesse} deals with the upper limit in the 
subdiffusive case $\lambda=0, \beta > \xi$. We comment here on the lower 
limit. In dimension $d \geq 3$, $n^{-1/2}|X_{[ns]}|$ converges to a  
 transient Bessel process, and it is not difficult to see that
$$ \limsup_{t \to \8}   \frac{\ln |Y_t|}{\ln t}= \lim_{t \to \8}   
\frac{\ln |Y_t|}{\ln t}
= \xi/(2\beta)$$
In dimension $d \leq 2$, $X$ is recurrent, and then $\liminf_{t}|Y_t|=0$ 
and
$$ \liminf_{t \to \8}  \frac{\ln |Y_t|}{\ln t}= -\8$$

(ii) A natural question is: What does the environment seen from the walker
look like in the subballistic case? In fact, the prominent feature is
that the size of surrounding cluster is essentially the largest
one which was visited so far. Consider for instance the case of positive
$\lambda$. One can prove that, for $\beta > \xi$ and $\epsilon >0$,
$$
\frac{1}{t}
\Big\vert \left\{ s \in [0,t]: (\ln t)^{-1}C_{Y_s} \in [\beta^{-1} -\epsilon,
\beta^{-1} + \epsilon]\right\} \Big\vert \longrightarrow 1
$$
$P$-a.s. as $t \nearrow \8$.
\medskip

\noindent {\bf Acknowlegement:} We thank Marina Vachkosvskaia for stimulating
discussions on the model.


\bigskip

\small
\begin{thebibliography}{99}
\bibitem{Aizenman-Barsky87} {\sc M.~Aizenman,  D.~Barsky} (1987)
Sharpness of the phase transition in percolation models.
{\it Comm. Math. Phys.} {\bf 108} 489--526. 

\bibitem{Berger-Gantert-Peres03}  {\sc N.~Berger, N.~Gantert, Y.~Peres} (2003)
The speed of biased random walk on percolation clusters.  
{\it Probab. Theory Related Fields} {\bf  126}  221--242.


\bibitem{Bouchaud92}  {\sc J.-P.~Bouchaud}  (1992)
 Weak ergodicity breaking and aging in disordered systems.
{\it J. Phys. I France} {\bf 2} 1705-1713

\bibitem{Benarous-Cerny05} {\sc G.~Ben Arous, J.~\u Cern\'y} (2006)
Dynamics of Trap Models.  Ch. 8 in
{\it  Mathematical Statistical Physics,}
(Les Houches  LXXXIII, 2005),
(Bovier, Dunlop, Van Enter, Den Hollander, Dalibard Ed.), Elsevier

\bibitem{Benarous-Cerny06} {\sc G.~Ben Arous, J.~\u Cern\'y} (2006)
Scaling limit for trap models on $\Z^d$.
{\it Ann. Probab.} to appear


\bibitem{Durrett86} {\sc R.~Durrett} (1986)
Multidimensional random walks in random environments with subclassical 
limiting behavior.
{\it Comm. Math. Phys.} {\bf 104} 87--102

\bibitem{Bramson-Durrett88}  {\sc M.~Bramson, R.~Durrett} (1988)
Random walk in random environment: a counterexample?
{\it Comm. Math. Phys.} {\bf 119} 199--211

\bibitem{Jacod-Shiryaev87} {\sc J.~Jacod, A.~Shiryaev} (1987)
{\it Limit theorems for stochastic processes.}
Grundlehren der Mathematischen Wissenschaften 288. Springer-Verlag, Berlin

\bibitem{Kesten-Kozlov-Spitzer75}  {\sc H.~Kesten, M.~Kozlov, F.~Spitzer}
 (1975)
A limit law for random walk in a random environment.  
{\it Compositio Math.} {\bf  30} 145--168

\bibitem{Menshikov86} {\sc M.~Menshikov} (1986)
 Coincidence of critical points in percolation problems.  
(Russian)  
{\it Dokl. Akad. Nauk SSSR} {\bf  288}  1308--1311.

\bibitem{Mathieu94}
{\sc P.~Mathieu} (1994)
Zero white noise limit through Dirichlet forms, with application to diffusions in a random medium. 
{\it Probab. Theory Related Fields} { \bf 99}  549--580

\bibitem{Meyn_Tweedie}  {\sc S.P.~Meyn, R.L.~TWEEDIE}
 (1993)
Markov chains and stochastic stability. 
{\it Communications and Control Engineering Series}{\it Springer-Verlag London Ltd.}

\bibitem{Monthus04} {\sc C.~Monthus} (2004)
Nonlinear response of the trap model in the aging regime: Exact results in the
strong-disorder limit
{\it Phys. Rev. E} {\bf 69}
 026103

\bibitem{Popov-Vachkovskaia05} {\sc S.~Popov, M.~Vachkovskaia} (2005)
Random walk attracted by percolation clusters.  
{\it Electron. Comm. Probab.} {\bf  10} 263--272

\bibitem{Revesz}  {\sc Rev\'esz}
{\it Random walk in random and non-random environments.}
 Second edition. World Scientific, Hackensack, NJ, 2005

\bibitem{Shen02} {\sc L.~Shen}  (2002)
Asymptotic properties of certain anisotropic walks in random media.  
{\it Ann. Appl. Probab.} {\bf  12}   477--510

\bibitem{Sinai82} {\sc Y.~Sinai} (1982)
 The limit behavior of a one-dimensional random walk in a random environment. 
(Russian)  
{\it Teor. Veroyatnost. i Primenen.} {\bf  27}  247--258

\bibitem{Solomon75} {\sc F.~Solomon}  (1975)
Random walks in a random environment.  
{\it Ann. Probability } {\bf 3}  1--31

\bibitem{Sznitman03}  {\sc A.-S.~Sznitman} (2003)
On the anisotropic walk on the supercritical percolation cluster.  
{\it Comm. Math. Phys.} {\bf  240}  123--148

\bibitem{ZeitouniSF}  {\sc O.~Zeitouni}
{\it Random walks in random environment. } 
Lectures on probability theory and statistics,  189--312,
Lecture Notes in Math. 1837, Springer, Berlin, 2004. 

\end{thebibliography}
\end{document}